\documentclass[11pt,a4paper]{amsart}

\usepackage{amsthm,amsmath,amssymb}
\usepackage[T1]{fontenc}
\usepackage[utf8]{inputenc}
\usepackage[english]{babel}
\usepackage{indentfirst}

\usepackage{graphicx}
\usepackage[dvipsnames]{xcolor}
\usepackage[pagebackref,breaklinks,colorlinks=true,linkcolor=MidnightBlue,citecolor=MidnightBlue]{hyperref}
\usepackage[shortlabels]{enumitem}
\usepackage{tikz}

\usepackage{mathFL}

\newcommand{\ent}{H_{\nu}}
\newcommand{\kl}{H}
\newcommand{\tilphi}{\widetilde{\phi}}
\newcommand{\tilpi}{\widetilde{\pi}}
\newcommand{\tilrho}{\widetilde{\rho}}
\theoremstyle{plain}
\newtheorem*{examples*}{Examples}

\title{A gradient descent perspective on Sinkhorn}
\date{\today}
\author{Flavien Léger}
\email{flavienleger@nyu.edu}

\begin{document}


\begin{abstract}
We present a new perspective on the popular Sinkhorn algorithm, showing that it can be seen as a Bregman gradient descent (mirror descent) of a relative entropy (Kullback--Leibler divergence). This viewpoint implies a new sublinear convergence rate with a robust constant.
\end{abstract}

\maketitle

\tableofcontents

\section{Introduction}
The Sinkhorn algorithm has been used to solve matrix scaling problems~\cite{yule,kruithof,demingstephan,bacharach} and in particular regularized optimal transport problems~\cite{wilson,erlander1980,erlanderstewart,galichonsalanie,cuturilightspeed}. Its convergence was studied in~\cite{sinkhorn1964,ruschendorf} and rates of convergence were first established in~\cite{franklinlorenz}.

The Sinkhorn algorithm can be seen as a solver for the minimum entropy problem
\[
H^*(\mu,\nu,R) = \inf_{\pi\in\Pi(\mu,\nu)} H(\pi|R),
\]
where $(X,\mu)$ and $(Y,\nu)$ are two probability spaces, $R$ is a measure on $X\times Y$ and $\Pi(\mu,\nu)$ denotes the space of probability measures on $X\times Y$ (sometimes called couplings or plans) having $X$-marginal $\mu$ and $Y$-marginal $\nu$. Moreover, $H$ is the relative entropy (also known as Kullback--Leibler divergence) defined by $H(\pi|R) = \iint\ln(\pi/R)\,\pi$. The Sinkhorn method constructs a sequence of couplings $\pi_0, \pi_{\frac{1}{2}}, \pi_1, \dots$ by alternative projections onto couplings with $Y$-marginal $\nu$ (these are $\pi_{\frac{1}{2}}, \pi_{\frac{3}{2}}$, etc) and couplings with $X$-marginal $\mu$ ($\pi_{1}, \pi_{2}$, etc).

Our contribution is a new perspective on the Sinkhorn algorithm. We show that it can be seen as a Bregman gradient descent (mirror descent) of the relative entropy $\rho\mapsto H(\rho|\nu)=\int\ln(\rho/\nu)\rho$. This allows us to derive a new sublinear convergence rate
\[
\kl(\rho_n|\nu) \le \frac{H^*(\mu,\nu,R)}{n},
\]
where $\rho_n$ is the $Y$-marginal of the iterate $\pi_n$. Contrary to all previously known global rates, our result features a robust constant $H^*$ which is always finite. In particular this new rate can be used for general reference measures $R$, without needing lower or upper bounds on the entries of $R$. We also obtain a new bound
\[
\kl(\rho_n|\nu) \le \frac{M_2(\mu)+M_2(\nu)}{n \,\eps }
\]
in the regularization of quadratic optimal transport 
\[
\inf_{\pi\in\Pi(\mu,\nu)} \iint \frac{1}{2}\abs{x-y}^2\,\pi(dx,dy) + \eps H(\pi|\mu\otimes \nu).
\]
Here $\eps>0$ and $M_2$ denotes second moments. This is of particular interest in the limit $\eps\to 0$.

In regard to convergence rates of the Sinkhorn algorithm, linear rates were obtained in~\cite{franklinlorenz} by using the so-called Hilbert projective metric. This elegant approach yields for instance bounds of the form $H(\rho_n|\nu) \lesssim \lambda^n$ for a constant $\lambda\in (0,1]$. In a large number of situations $\lambda$ is away from $1$ and this rate is much stronger than our new sublinear rate. However all the known linear rates deteriorate when the reference measure $R$ (i.e. the ``matrix'' we wish to scale) contains large or small (nonnegative) entries. As an example $\lambda=1$ if $R$ contains zero entries, in which case the linear rate is unusable. Therefore there is a dichotomy where either $R$ has good lower and upper bounds, in which case fast linear rates exist, or $R$ contains small or large values, in which case Hilbert metric theory might not even imply convergence of the iterates. Note that in many cases one might be interested in scaling matrices $R$ which contain many zeros, or are even sparse. Our new results remedy this situation by providing a convergence rate which is applicable to any problem.

Let us now mention some related works. In~\cite{altschuler,chakrabarty,dvurechensky} the authors derive sublinear estimates for the relative entropy $H(\nu|\rho_n)$. Our result improves on these estimates by obtaining an inequality $\kl(\rho_n|\nu) \le H^* / n$, and more importantly by identifying the robust constant $H^*$. Indeed the constants appearing in these papers all contain a $-\ln(\min_{ij}R_{ij})$ term which blows up as $\min_{ij}R_{ij}\to 0$ (their setting is finite-dimensional). In~\cite{mishchenko} a different mirror descent approach to the Sinkhorn scheme was presented: their proposed alternating mirror descent seemingly doesn't allow to derive convergence results.
In~\cite{mensch} an online variant of the Sinkhorn scheme is introduced using a block-convex stochastic mirror
descent method.

\pagebreak

\section{Background}\label{sec:background}
\subsection{Bregman divergences} \label{sec:bregmandiv}

\begin{definition}
Consider a differentiable function $F\colon\RN\to\R$. Its \emph{Bregman divergence} is defined by
\[
F(\phi_2|\phi_1) = F(\phi_2) - F(\phi_1) - \bracket{F'(\phi_1),\phi_2-\phi_1},
\]
for any $\phi_1,\phi_2\in\RN$. Here $F'$ denotes the derivative (or gradient) of $F$, i.e. the vector $F'(\phi)=(\partial_{i}F(\phi))_i$, and $\bracket{\cdot,\cdot}$ is the usual dot product.
\end{definition}

Let us gather below some well-known results in the theory of Bregman divergences.

\begin{prop}\label{prop:bregmandivergence}
Let $F\colon\RN\to\R$ be a convex and differentiable function, and denote by $F^*$ its convex conjugate $F^*(\rho)=\sup_{\phi} \bracket{\phi, \rho} -F(\phi)$. Then
\begin{enumerate}[(i)]

\item \label{prop:bregmandivergence:switch} $F(\phi_2|\phi_1)=F^*(\rho_1|\rho_2)$ for all $\phi_1,\phi_2\in\RN$, where we set $\rho_i=F'(\phi_i)$. 

\item \label{prop:bregmandivergence:chainrule} Fix $a\in\RN$ and define $F_a(\phi)=F(\phi|a)$. Then $F_a(\phi_2|\phi_1)=F(\phi_2|\phi_1)$. 
\end{enumerate}
\end{prop}

\subsection{Bregman gradient descent}\label{sec:gradientdescent}
Consider a differentiable function $H\colon\RN\to\R$ that we wish to minimize without constraints. Let $G\colon\RN\to\R$ be a differentiable strictly convex function which will be used as a movement limiter. The gradient descent iteration with a Bregman divergence based on $G$, also called \emph{mirror descent}~\cite{nemirovskyyudin,beckteboullemirror}, takes the form
\begin{equation} \label{eq:generalgradientdescent}
\rho_{n+1} = \argmin_\rho H(\rho_n) +  \bracket{H'(\rho_n),\rho-\rho_n} + G(\rho|\rho_n),
\end{equation}
for all $n\ge 0$, where $H'$ denotes the derivative of $H$ (see previous section). The optimality conditions are given by 
\[
G'(\rho_{n+1}) - G'(\rho_n) = -H'(\rho_n).
\]
The following result gathers well-known facts in first-order optimization theory.

\begin{theorem}[Unconstrained gradient descent]\label{thm:gradientdescent}
Consider the gradient descent method~\eqref{eq:generalgradientdescent} under the previous hypotheses on $H$ and $
G$. 
\begin{enumerate}[i)]
\item If the objective function is dominated by the movement limiter, i.e. $H(\tilrho|\rho)\le G(\tilrho|\rho)$ for all $\rho,\tilrho$, then we have the descent property
\[
H(\rho_{n+1}) \le H(\rho_n) - G(\rho_n|\rho_{n+1}),
\]
for all $n\ge 0$.

\item If in addition $H$ is convex then we have the convergence rate $H(\rho_n) \le \inf_{\rho} H(\rho) + \frac{G(\rho|\rho_0) }{n}$, for all $n\ge 1$. Therefore if $H$ admits a minimizer $\nu$ then 
\[
H(\rho_n)-H(\nu) \le \frac{G(\nu|\rho_0) }{n}.
\]

\end{enumerate}
\end{theorem}

We include for the reader's convenience a short proof of Theorem~\ref{thm:gradientdescent}.

\begin{proof}
Fix an iteration $n\ge 0$ and define the convex function
\[
\xi(\rho)=H(\rho_n)+\bracket{H'(\rho_n),\rho-\rho_n}+G(\rho|\rho_n).
\]
By~\eqref{eq:generalgradientdescent} we have $\xi'(\rho_{n+1})=0$ and therefore  $\xi(\rho_{n+1}) = \xi(\rho)-\xi(\rho|\rho_{n+1})$ for any $\rho$. Next we use the bound $H(\rho_{n+1}|\rho_n)\le G(\rho_{n+1}|\rho_n)$ together with the identity  $\xi(\rho) = H(\rho)-H(\rho|\rho_n)+G(\rho|\rho_n)$ evaluated at $\rho=\rho_{n+1}$. We obtain
\[
H(\rho_{n+1}) \le \xi(\rho_{n+1}) = \xi(\rho)-\xi(\rho|\rho_{n+1}), 
\]
for any $\rho$. We use this inequality to prove points i) and ii). 
\begin{enumerate}[i)]
\item 
Take $\rho=\rho_n$. By Prop.~\ref{prop:bregmandivergence}\ref{prop:bregmandivergence:chainrule} and linearity of the Bregman divergence operation we have $\xi(\rho_n|\rho_{n+1})=0+0+G(\rho_n|\rho_{n+1})$. Therefore $H(\rho_{n+1})\le \xi(\rho_n)-\xi(\rho_n|\rho_{n+1})=H(\rho_n)-G(\rho_n|\rho_{n+1})$. 
\item
Convexity on $H$ now implies the upper bound $\xi(\rho) \le H(\rho)+G(\rho|\rho_n)$. Therefore $H(\rho_{n+1}) \le H(\rho)+G(\rho|\rho_n)-\xi(\rho|\rho_{n+1})=H(\rho)+G(\rho|\rho_n)-G(\rho|\rho_{n+1})$. To conclude, sum this last inequality from $0$ to $n-1$ and use decrease of $H(\rho_n)$ to obtain $n H (\rho_n)\le nH(\rho)+G(\rho|\rho_0)-G(\rho|\rho_n) \le nH(\rho)+G(\rho|\rho_0)$.
\end{enumerate}
\end{proof}

\subsection{Entropic regularization of optimal transport}
Let $(X,\mu)$ and $(Y,\nu)$ be two probability spaces and consider a cost function $c\colon X\times Y\to\R$. We are interested in the regularized optimal transport problem~\cite{galichonsalanie,cuturilightspeed,peyrecuturibook}
\begin{equation}\label{eq:entropicminimization}
\inf_{\pi\in\Pi(\mu,\nu)} \iint_{X\times Y} c(x,y)\,\pi(dx,dy) + \eps \iint_{X\times Y} \ln\Big(\frac{\pi(dx,dy)}{\mu(dx)\nu(dy)}\Big)\,\pi(dx,dy),
\end{equation}
where $\eps>0$. Here $\Pi(\mu,\nu)$ denotes the set of couplings $\pi$ having $X$-marginal $\mu$ and $Y$-marginal $\nu$, i.e. $\int_Y\pi(dx,dy)=\mu(dx)$ and $\int_X\pi(dx,dy)=\nu(dy)$. The above problem can be written as
\begin{equation*}
\inf_{\pi\in\Pi(\mu,\nu)} \eps H(\pi|R),
\end{equation*}
by defining $R(dx,dy) = e^{-c(x,y)/\eps}\mu(dx)\nu(dy)$. The relative entropy is defined by $H(\pi|R) = \iint_{X\times Y} \ln(\pi/R)\,d\pi$ when $\pi$ is absolutely continuous with respect to $R$, and $\infty$ otherwise. We will focus primarily on the dual formulation of~\eqref{eq:entropicminimization}, which takes the form
\[
\eps\,\sup_{\phi,\psi} \int_Y \phi(y)\,\nu(dy) + \int_X \psi(x)\,\mu(dx) - \ln\left(\iint_{X\times Y} \!\!\!\!\!e^{\phi(y)+\psi(x)}\,R(dx,dy)\right).
\]

The supremum is here taken over functions $\phi\colon Y\to\R$ and $\psi\colon X\to\R$.

\section{Sinkhorn as a gradient descent method}

\subsection{Definitions and notations
}
Let $(X,\mu)$ and $(Y,\nu)$ be two probability spaces, and $R$ a reference measure on $X\times Y$ (note that we do not assume that $R$ has necessarily mass $1$). Define the dual functional
\[
D(\phi,\psi)=\bracket{\phi,\nu} - \bracket{\psi,\mu} - \ln\Big(\iint e^{\phi(y)-\psi(x)}R(dx,dy)\Big),
\]
over functions $\phi\colon Y\to\R$ and $\psi\colon X\to\R$. 
Here $\bracket{\phi,\nu}=\int_Y\phi\,d\nu$ and $\bracket{\psi,\mu}=\int_X\psi\,d\mu$.

The Sinkhorn method can be seen as an iterative solver for the optimization problem $\sup_{\phi,\psi} D(\phi,\psi)$ (note that this is a concave maximization problem). We can write the Sinkhorn iterations concisely by making use of two transforms $\phi\to\phi^+$ and $\psi\to\psi^-$, defined by 
\begin{equation}\label{eq:deftransforms}
\begin{aligned}
\phi^+(x) &= \ln\Big(\int_Y e^{\phi(y)}\,R(dx,dy)/\mu(dx)\Big),\\
\psi^{-}(y) &= -\ln\left(\int_X e^{-\psi(x)}\,R(dx,dy)/\nu(dy)\right).
\end{aligned}
\end{equation}
The fractions in the expression above should be interpreted in the Radon--Nikodym sense and are assumed to be well-defined. Then, the Sinkhorn iteration takes the form
\begin{equation}\label{eq:sinkhorniteration}
\phi_{n+1}=(\phi_n)^{+-}.
\end{equation}
Note that the ``$+$''-transform maps a potential $\phi$ defined on $Y$ to a potential $\phi^+$ defined on $X$ (and vice versa for the ``$-$''-transform). These two transformations play similar roles as the $c$-transforms from optimal transport~\cite{Villanibook2}. 

To each pair of potential $(\phi,\psi)$ is associated a primal quantity: the coupling or ``plan''
\begin{equation} \label{eq:defcoupling}
\pi(\phi,\psi)(dx,dy) = Z^{-1} e^{\phi(y)-\psi(x)}R(dx,dy),
\end{equation}
where the scalar 
$Z=\iint e^{\phi(y)-\psi(x)}R(dx,dy)$
ensures that the measure $\pi(\phi,\psi)$ has mass $1$. In this paper we will consider plans $\pi(\phi,\phi^+)$ or $\pi(\psi^-,\psi)$ and therefore $Z$ will always be $1$. 

In order to relate potentials and densities let us first define
\[
J(\phi) = \sup_{\psi} D(\phi,\psi),
\]
where the supremum is taken over all function $\psi\colon X\to\R$. Since $D$ is concave, $J$ is easily seen to be concave as well. Second, we define a functional $F$ by $J(\phi) = \bracket{\phi,\nu} - F(\phi)$. Written more explicitly, we have
\[
F(\phi)=\bracket{\phi^+\!,\mu}.
\]
The important role played by the convex functional $F$ lies in its derivative $F'$ which is a bridge between potentials and densities, as shown by the following result.
\begin{lemma}\label{lemma:derivativeF}
Consider a potential $\phi\colon Y\to\R$. Then $F'(\phi) = p_Y\pi(\phi,\phi^+)$,where $p_Y$ denotes the $Y$-marginal projection. In other words,
\[
F'(\phi)(dy)=\int_X e^{\phi(y)-\phi^+(x)}\,R(dx,dy).
\]
The $X$-marginal (sum along the rows) of $\pi(\phi,\phi^+)$ is always $\mu$.
\end{lemma}
In the above lemma and in the rest of this note we use $F'$ to denote the derivative (or first variation) of $F$. It is defined for instance by $F'(\phi)h=\lim_{\eps\to 0} \big(F(\phi+\eps h)-F(\phi)\big)/\eps$. 

\begin{proof}[Proof of Lemma~\ref{lemma:derivativeF}]
The result of this lemma can be obtained from an elementary computation of derivative, using the expression
\[
F(\phi)=\bracket{\phi^+,\mu}=\int_X \ln\Big(\int_Y e^{\phi(y)}R(dx,dy) / \mu(dx)\Big)\,\mu(dx).
\]
\end{proof}

\subsection{Main results}
We recall that $(X,\mu)$ and $(Y,\nu)$ are two probability spaces and that $R$ is a measure on $X\times Y$. In the previous section, we defined two transformations ``$+$'' and ``$-$'' by~\eqref{eq:deftransforms}, a coupling function $\pi(\phi,\psi)$ by~\eqref{eq:defcoupling} and a functional $F(\phi)=\bracket{\phi^+\!,\mu}$.

Our starting point is the following observation, already present in~\cite{berman}
\begin{lemma}[\cite{berman}]\label{prop:Sinkhorndensities}
the Sinkhorn iteration $\phi_{n+1}=(\phi_n)^{+-}$ can be written as
\begin{equation}\label{eq:sinkhorniterationrho}
\phi_{n+1}-\phi_n = -\ln(\rho_n / \nu),
\end{equation}
where $\rho_n$ denotes the probability measure associated with $\phi_n$, i.e. $\rho_n=F'(\phi_n)$ (see Lemma~\ref{lemma:derivativeF}).
\end{lemma}
\begin{proof}
Let $n\ge 0$ and consider a Sinkhorn iterate $\phi_n$. Set $\psi_n=(\phi_n)^+$, so that $\phi_{n+1}=(\psi_n)^-$. Let $\rho_n$ denote the $Y$-marginal of $\pi(\phi_n,\psi_n)$, i.e.
\[
\rho_n(dy) = \int_X e^{\phi_n(y)-\psi_{n}(x)}\,R(dx,dy).
\]
This can be written, recognizing the ``$-$''-transform, as
\[
\rho_n(dy)=e^{\phi_n(y)-\phi_{n+1}(y)}\,\nu(dy),
\]
which implies the desired equality.
\end{proof}

Then the main result of this paper says that the Sinkhorn iteration~\eqref{eq:sinkhorniterationrho} can be seen as a gradient descent method of a relative entropy (Kullback--Leibler divergence):

\begin{theorem} \label{thm:main}
Let $\pi_n=\pi(\phi_n,\phi_n^+)$ be the coupling produced by the Sinkhorn iteration $\phi_n\to\phi_{n+1}$ (see~\eqref{eq:sinkhorniteration} and~\eqref{eq:sinkhorniterationrho}) and denote by $\rho_n$ its $Y$-marginal (we recall that the $X$-marginal of $\pi_n$ is always $\mu$). Then the Sinkhorn scheme can be seen as the gradient descent
\begin{equation}\label{eq:bregmangradientdescent}
(F^*)'(\rho_{n+1}) - (F^*)'(\rho_n) = -\, H'_{\nu}(\rho_n).
\end{equation}
Here $H_{\nu}(\rho) = H(\rho|\nu)=\int\ln(\rho/\nu)\rho$ denotes the relative entropy of $\rho$ with respect to $\nu$, and $F^*$ is the convex conjugate of $F$. Moreover $~'$ denotes derivative, see~\eqref{lemma:derivativeF}. 
\end{theorem}
\begin{proof}
By Lemma~\ref{prop:Sinkhorndensities} we can write a Sinkhorn step as 
\[
\phi_{n+1}-\phi_n = -H_\nu'(\rho_n).
\]
By Lemma~\ref{lemma:derivativeF} we know that $F'(\phi_n)=\rho_n$. Convex conjugation inverts derivatives, therefore $\phi_n=(F^*)'(\rho_n)$.
\end{proof}

As a consequence we derive a $O(1/n)$ convergence rate.

\begin{corollary}[Sublinear rate] \label{cor:rate}
Let $H^*(\mu,\nu,R)=\inf_{\pi\in\Pi(\mu,\nu)}H(\pi|R)$ be the value of the minimum entropy problem. Assume that $R$ has total mass $1$. Then the gradient descent formulation~\eqref{eq:bregmangradientdescent} implies decrease of the relative entropies $\kl(\rho_{n+1}|\nu) \le \kl(\rho_n|\nu)$, and a sublinear convergence rate with a robust constant,
\begin{equation}\label{eq:sublinearrate}
H(\rho_n|\nu) \le  \frac{H^*(\mu,\nu,R)}{n}
\end{equation}
for all $n\ge 1$. In particular this bound is finite (whenever there exists a solution to the minimum entropy problem).
\end{corollary}

Before proving Corollary~\ref{cor:rate} let us point out that general measures $R$ which don't necessarily sum up to $1$ can be handled. In that case~\eqref{eq:sublinearrate} should be replaced by
\begin{equation}\label{eq:sublinearrategeneralR}
H(\rho_n|\nu) \le  \frac{H^*(\mu,\nu,R)+\ln\big(\iint R\big)}{n}.
\end{equation}
In fact a slightly stronger bound valid for any measure $R$ is implied by the gradient descent viewpoint, namely
\[
H(\rho_n|\nu) \le  \frac{H^*(\mu,\nu,R)-H(\mu|\bar{\mu})}{n},
\]
where $\bar{\mu}$ is the $X$-marginal of $R$. Note that when $\bar{\mu}$ is not a probability measure the entropy $H(\mu|\bar{\mu})$ can be positive or negative.

The proof of Corollary~\ref{cor:rate} relies on two lemmas. The first one says that the movement limiter based on $F^*$ can be expressed as a relative entropy over couplings.

\begin{lemma}\label{lemma:bregmandivergenceFstar}
Let $\phi$ and $\tilphi$ be two potentials defined over $Y$. Denote $\pi=\pi(\phi,\phi^+)$, $\tilpi=\pi(\tilphi,\tilphi^+)$, and set the $Y$-marginals $\rho=p_Y\pi$ and $\tilrho=p_Y\tilpi$. Then
\[
F(\phi|\tilphi)=F^*(\tilrho|\rho) = H(\tilpi|\pi),
\]
with $H(\tilpi|\pi)=\iint\ln(\tilpi/\pi)\,\tilpi$.
\end{lemma}

\begin{proof}
Let $\phi$ and $\tilphi$ be two potentials on $Y$, and denote $\pi$ and $\tilpi$ the corresponding couplings, as well as $\rho$ and $\tilrho$ the corresponding probability measures on $Y$. 

Firstly, the identity $F(\phi|\tilphi)=F^*(\tilrho|\rho)$ is a general property of Bregman divergences, see~Prop.~\ref{prop:bregmandivergence}\ref{prop:bregmandivergence:switch}. Here it follows from Lemma~\ref{lemma:derivativeF} which says that $\rho=F'(\phi)$ and $\tilrho=F'(\tilphi)$.

Secondly, we prove that $F(\phi|\tilphi)=\kl(\tilpi|\pi)$. We write
\[
\kl(\tilpi|\pi) = \iint\ln\Big(\frac{\tilpi}{\pi}\Big)\,\tilpi.
\]
Using the expression $\pi(dx,dy)=e^{\phi(y)-\phi^+(x)}\,R(dx,dy)$ and the corresponding one for $\tilpi$ we obtain
\begin{align*}
\kl(\tilpi|\pi) &= \iint \Big[\big(\tilphi(y)-\phi(y)\big) - \big(\tilphi^+(x)-\phi^+(x)\big)\Big]\,\tilpi(dx,dy) \\
 &= \bracket{\tilphi-\phi, \tilrho} - \bracket{\tilphi^+-\phi^+, \mu},
\end{align*}
since the $X$-marginal of the couplings $\pi(\phi,\phi^+)$ we construct is always $\mu$. Continuing, 
\begin{align*}
\kl(\tilpi|\pi) &= \bracket{\tilphi-\phi, F'(\tilphi)} - F(\tilphi) + F(\phi) \\
&= F(\phi|\tilphi).
\end{align*}
\end{proof}

The next lemma says that the objective function is bounded in a convex sense by the movement limiter (it is ``$1$-smooth'' in the language of first-order optimization).

\begin{lemma}\label{lemma:movementlimiterbound}
For all probability measures $\rho, \tilrho$ on $Y$,
\[
H_{\nu}(\tilrho|\rho) \le F^*(\tilrho|\rho),
\]
where $H_{\nu}(\rho) = H(\rho|\nu)$.
\end{lemma}
\begin{proof}

To show this, first use Lemma~\ref{lemma:bregmandivergenceFstar} to write
\[
F^*(\tilrho|\rho) =  \kl(\tilpi|\pi),
\]
where $\pi$ and $\tilpi$ are defined in accordance with Lemma~\ref{lemma:bregmandivergenceFstar}. Then we use a property of the relative entropy (true more generally for $f$-divergences)  that relative entropy decreases when taking marginals, thus
\[
\kl(\tilpi|\pi) \ge \kl(p_Y\tilpi|p_Y\pi) .
\]
This property is a simple consequence of Jensen's inequality and is left as an exercise to the reader. We have obtained
\[
F^*(\tilrho|\rho) \ge  \kl (\tilrho|\rho).
\]
To conclude we use Prop.~\ref{prop:bregmandivergence}\ref{prop:bregmandivergence:chainrule} to say that $\kl(\tilrho|\rho)=H_{\nu}(\tilrho|\rho)$.
\end{proof}

We are now able to prove the convergence rate.

\begin{proof}[Proof of Corollary~\ref{cor:rate}]
The crucial ingredient needed to derive a $O(1/n)$ convergence rate for a gradient descent scheme is showing that the movement limiter dominates (in a convex sense) the objective function. We refer to Theorem~\ref{thm:gradientdescent} in Section~\ref{sec:gradientdescent} for a precise statement. For the problem at hand, this is precisely the content of Lemma~\ref{lemma:movementlimiterbound}, $H_\nu(\tilde{\rho}|\rho) \le F^*(\tilde\rho|\rho)$. Thus we immediately obtain
\[
H(\rho_n|\nu) \le \frac{F^*(\nu|\rho_0)}{n},
\]
for all $n\ge 1$. We have therefore derived the desired $O(1/n)$ convergence rate. We would now like to obtain a more tractable inequality. To this end, assume that the initial iterate $\phi_0$ is identically zero. Let $\pi_0$ be the coupling associated to $\rho_0$ and let $\pi^*$ be the coupling associated to $\nu$, i.e. $\pi^*$ is the minimizer to the entropic problem $\inf_{\pi} H(\pi|R)$ (we assume in this paper that the minimizer exists). By Lemma~\ref{lemma:bregmandivergenceFstar} we know that $F^*(\nu|\rho_0)=H(\pi^*|\pi_0)$. Denote $\psi_0=(\phi_0)^+$; then $\pi_0(dx,dy) = e^{\phi_0(y)-\psi_0(x)}R(dx,dy)$ and we have
\begin{align*}
H(\pi^*|\pi_0) &= \iint \ln\Big(\frac{\pi^*}{\pi_0}\Big)\,\pi^*\\
 &=\iint \ln\Big(\frac{\pi^*}{R}\Big)\,\pi^* - \iint \phi_0 \,\pi^* + \iint \psi_0 \,\pi^* \\ 
 &= H(\pi^*|R) - \bracket{\phi_0,\nu} + \bracket{\psi_0,\mu}.
\end{align*}
Since we assume that $\phi_0=0$ the second term cancels, and the third term is 
\[
\bracket{\psi_0,\mu} = \int \ln\Big(\int e^0 R/\mu\Big)\mu = -H(\mu|\bar{\mu}),
\]
where $\bar{\mu}$ is the $X$-marginal of $R$. 
Therefore $H(\pi^*|\pi_0) = H(\pi^*|R) - H(\mu|\bar{\mu})$. If $R$ has total mass $1$ then so does its marginal $\bar{\mu}$, which implies that the relative entropy $H(\mu|\bar{\mu})$ is nonnegative. Thus $H(\pi^*|\pi_0) \le H(\pi^*|R)$ which concludes the proof.

\end{proof}

We develop below discussions and examples related to these results.

\subsubsection*{Discussion on the gradient descent formulation}
A short introduction on Bregman divergences and gradient descent methods is contained in Section~\ref{sec:background}. 

Our gradient descent perspective in Theorem~\ref{thm:main} shifts the focus of the Sinkhorn method from potentials to probability measures. It is based on the ``semi-dual'' formulation~\eqref{eq:sinkhorniterationrho} which eliminates one of the two potentials (here $\psi$) and provides a description of the Sinkhorn algorithm based only on $Y$-variables ($\phi$ and $\rho$). By symmetry it is possible of course to state an analogue of Theorem~\ref{thm:main} using instead variables defined on $X$.

A rather nonstandard aspect of the theorem's gradient scheme is the movement limiter based on $F^*$. First recall from Section~\ref{sec:gradientdescent} that the gradient descent update~\eqref{eq:bregmangradientdescent} admits the variational formulation
\[
\rho_{n+1}=\argmin_{\rho}H_{\nu}(\rho_n)+\bracket{H_{\nu}'(\rho_n),\rho-\rho_n} + F^*(\rho|\rho_n),
\]
which highlights the form of movement limiter $F^*(\rho|\rho_n)$: a Bregman divergence based on the function $F^*$. Here $F^*$ is specific to the the problem at hand; from the optimization point of view it is natural to have movement limiters well-adapted to the objective function. The result of Lemma~\ref{lemma:bregmandivergenceFstar} might shed some light on this Bregman divergence by expressing it as a relative entropy (Kullback--Leibler divergence) of the corresponding couplings.

A benefit of a gradient descent framework is that obtaining a convergence rate becomes a clearly defined problem: the movement limiter should dominate (in the convex sense) the objective function. Here it means roughly speaking obtaining the inequality over Hessians
\[
\ent'' \le (F^*)''.
\]
(Note that we don't actually need these functions to be twice-differentiable).  This is proven in Lemma~\ref{lemma:movementlimiterbound} and relies on the following simple fact: the relative entropy decreases when taking marginals, thus
\[
\kl(\tilpi|\pi) \ge \kl(p_Y\tilpi|p_Y\pi) .
\]

\subsubsection*{Discussion on the convergence rate}

The strength of our convergence rate
\[
H(\rho_n|\nu) \le \frac{H^*}{n}
\]
lies in the robust constant $H^*$ rather than its sublinear nature, since \emph{linear} rates are well-known to exist (as discussed in the next paragraph). Indeed, the constant $H^*=H^*(\mu,\nu,R)$ is finite as soon as the feasibility set of the entropic problem~\eqref{eq:entropicminimization} is non-empty. In other words, when there is a solution then the convergence rate $H^* / n$ holds. For instance, this allows to deal with reference measures $R$ with zero entries. To the best of our knowledge this improves on all the known global rates for the Sinkhorn algorithm which are sensitive to zero entries of $R$.

A classical approach to obtain rates on the convergence of the Sinkhorn method is to use the Hilbert projective metric~\cite{franklinlorenz}. Then one can derive linear convergence rates of the form $\kl(\rho_n|\nu) \lesssim \lambda^{-n} $ for some $\lambda\in (0,1]$; however the constant $\lambda$ can be weak in practice. We refer to~\cite{peyrecuturibook} for precise formulas but let us point out that $\lambda\to 1$ as $\min_{ij}R_{ij}\to 0$.

More recently, a series of work~\cite{altschuler,chakrabarty,dvurechensky} have derived sublinear estimates for the relative entropy $H(\nu|\rho_n)$ in the same spirit as our convergence rate. In these works is proven, roughly speaking, that $O(1/\delta)$ iterations are needed to obtain an accuracy of $\delta>0$, measured in a KL divergence. The convergence rate obtained from our gradient descent viewpoint improves on these estimates on two fronts. First we obtain that the quantities $\kl(\rho_n|\nu)$ decrease as $n$ grows, as well as a true inequality $\kl(\rho_n|\nu) \le \frac{H^*}{n}$. Second the constants appearing in the literature slightly differ from one another but all have in common a $-\ln(\min_{ij}R_{ij})$ term which blows up as $\min_{ij}R_{ij}\to 0$ (their setting is finite-dimensional so that $R$ is a matrix with entries $R_{ij}$).

We now present some examples which allow a more explicit bound on the rate constant $H^*$. 

\begin{example}[Regularization of quadratic optimal transport]\label{example:quadot}
Take $X=Y=\Rd$ and let $\mu$ and $\nu$ be two probability measures on $\Rd$ with finite second moments, 
\[
M_2(\mu)=\int \abs{x}^2\,\mu(dx) < \infty, \quad M_2(\nu)=\int \abs{y}^2\,\nu(dy) < \infty.
\]
Fix $\eps >0$ and consider the problem 
\[
E^*=\inf_{\pi} \iint \frac{1}{2}\abs{x-y}^2\,\pi(dx,dy) + \eps \iint \ln\Big(\frac{\pi(dx,dy)}{\mu(dx)\nu(dy)}\Big)\,\pi(dx,dy).
\]
As usual the infimum runs over couplings $\pi\in\Pi(\mu,\nu)$. To fit into the framework of this paper define $R(dx,dy)=e^{-\frac{\abs{x-y}^2}{2\eps}}\mu(dx)\nu(dy)$. Then $E^*=\eps\inf_{\pi} H(\pi|R)$. The simple upper bound $E^*\le M_2(\mu)+M_2(\nu)$ can be obtained with $\pi(dx,dy)=\mu(dx)\nu(dy)$. Also note that the total mass of $R$ satisfies $\ln\big(\iint R\big) \le 0$. Using the general form~\eqref{eq:sublinearrategeneralR} of our main result we obtain the Sinkhorn convergence rate
\begin{equation}
H(\rho_n|\nu) \le \frac{M_2(\mu)+M_2(\nu)}{n \,\eps },    
\end{equation}
for all $n\ge 1$.

Often one is interested in the limit $\eps\to 0$. Then the above inequality provides a $O\Big(\frac{1}{n\eps}\Big)$ bound which can be compared to the $O\big((1-e^{-1/\eps})^{2n}\big)$ bound from the Hilbert metric theory~\cite{peyrecuturibook}. For instance, with $\eps=10^{-4}$, assuming all the other constants are $O(1)$, one can guarantee an accuracy $H(\rho_n|\nu)<10^{-3}$ in
\begin{itemize}[\textendash]
    \item $n\sim 10^7$ iterations with our $O\Big(\frac{1}{n\eps}\Big)$ bound; and
    \item $n\sim e^{(10^4)}$ iterations with a $O\big((1-e^{-1/\eps})^{2n}\big)$ bound.
\end{itemize}
\end{example}

\begin{example}[Entropic Talagrand inequality]\label{example:talagrand}
Take $X=Y=\Rd$, let $\mu$ and $\nu$ be two probability measures absolutely continuous with respect to the Lebesgue measure and let $R$ be the joint measure at times $0$ and $T$ associated with the SDE 
\[
dX_t=-\nabla U(X_t)\,dt + dW_t, \quad X_0\sim m(dx):=e^{-U(x)}\,dx.
\]
We assume that the potential energy $U$ is normalized and satisfies the strong convexity bound $D^2U(x) \ge \lambda I$ for some $\lambda>0$. Here $I$ denotes the $d\times d$ identity matrix. One can have in mind for instance the Ornstein--Uhlenbeck process corresponding to $U(x)=\frac{\lambda}{2}\abs{x}^2$. 

In this setting we can obtain more precise bounds for our Sinkhorn convergence rate~\eqref{eq:sublinearrate} by using recent results in~\cite{conforti,confortitamanini}.  These works provide an entropic version of the Talagrand inequality from optimal transport. Specifically, they obtain the following bound on the entropic cost: $H^*(\mu,\nu,R) \le \frac{H(\mu|m) + H(\nu|m)}{1-e^{-\lambda T}}$, where $m(dx)=e^{-U(x)}\,dx$. For our purposes, this implies the Sinkhorn convergence rate
\[
H(\rho_n|\nu) \le \frac{H(\mu|m) + H(\nu|m)}{ n \,(1-e^{-\lambda T}) }.
\]
\end{example}

\section*{Acknowledgements}
The author is grateful to Gabriel Peyré for helpful discussions.

\bibliographystyle{amsalphaurl}
\bibliography{sinkhorn}

\end{document}